\documentclass[11pt,reqno]{amsart}

\usepackage{amssymb, amscd, amsmath, amsthm, float, graphicx, longtable,color,
  enumerate, fullpage%, bm, stmaryrd
}
\usepackage[textwidth=2cm,textsize=small]{todonotes}
\usepackage{mathrsfs}
\usepackage{psfrag}
\usepackage[all]{xy}
\SelectTips{eu}{} 
\xyoption{curve}

\usepackage[utf8]{inputenc}    
\usepackage[T1]{fontenc}       

\newcommand{\hide}[1]{}

\usepackage[colorlinks]{hyperref}
\DeclareMathOperator{\add}{add{}}

\DeclareMathOperator{\Ann}{Ann}
\DeclareMathOperator{\charac}{char}

\DeclareMathOperator{\cok}{cok}

\DeclareMathOperator{\End}{End}
\DeclareMathOperator{\Ext}{Ext}

\DeclareMathOperator{\GP}{GP}

\DeclareMathOperator{\Hom}{Hom}

\DeclareMathOperator{\id}{id}

\DeclareMathOperator{\im}{im}

\DeclareMathOperator{\MCM}{MCM}
\DeclareMathOperator{\MF}{MF}

\DeclareMathOperator{\op}{op}
\DeclareMathOperator{\pd}{projdim}%%% or better ``pd'' as before?

\DeclareMathOperator{\uMCM}{\underline{\MCM}}

\newcommand{\rmod}[1]{\mbox{mod-}#1}
\newcommand{\lmod}[1]{#1\mbox{-mod}}
\newcommand{\rMod}[1]{\mbox{Mod-}#1}
\newcommand{\lMod}[1]{#1\mbox{-Mod}}

\renewcommand{\phi}{\varphi}

\renewcommand{\bar}{\overline}
\renewcommand{\tilde}{\widetilde}

\renewcommand{\mod}{\operatorname{mod}}
\renewcommand{\leq}{\leqslant}

\renewcommand{\geq}{\geqslant}

\renewcommand{\to}{\longrightarrow}

\newcommand\xto{\ -\negmedspace\negmedspace\negthinspace\xrightarrow}

\newcommand{\gen}[1]{\langle #1 \rangle}

\newcommand{\EE}{{\mathbb E}}

\newcommand{\MM}{{\mathbb M}}

\newcommand{\PP}{{\mathbb P}}

\newcommand{\TT}{{\mathbb T}}

\newcommand{\cala}{{\mathcal A}}

\newcommand{\cali}{{\mathcal I}}

\newcommand{\calt}{{\mathcal T}}

\newcommand{\calx}{{\mathcal X}}

\theoremstyle{plain}
\newtheorem{theorem}{Theorem}

\newtheorem{prop}[theorem]{Proposition}
\newtheorem{proposition}[theorem]{Proposition}

\newtheorem{lemma}[theorem]{Lemma}

\newtheorem{corollary}[theorem]{Corollary}
\newtheorem{question}[theorem]{Question}

\newtheorem*{conjecture*}{Conjecture}
\newtheorem*{theorem*}{Theorem}
\newtheorem*{prop*}{Proposition}

\theoremstyle{definition}

\newtheorem{definition}[theorem]{Definition}

\newtheorem{remark}[theorem]{Remark}
\newtheorem*{notation*}{Notation}

\newtheorem{example}[theorem]{Example}

\numberwithin{theorem}{section}
\numberwithin{equation}{section}

\begin{document}

\title[Theorems of Kn\"orrer and Herzog-Popescu]{%
Some extensions of theorems of Kn\"orrer and Herzog-Popescu}

\author[]{Alex S.\ Dugas}
\email{adugas@pacific.edu}
\address{Dept.\ of Mathematics, University of the Pacific,
  Stockton, California 95211}

\author[G.J. Leuschke]{Graham J.\ Leuschke}
\email{gjleusch@syr.edu}
\address{Dept.\ of Mathematics, Syracuse University,
Syracuse NY 13244, USA}
\urladdr{\href{http://www.leuschke.org/}{http://www.leuschke.org/}}

\thanks{}

\date{\today}

%  \subjclass[2010]{Primary: 
% %   13C14, % MCM modules
% %   14A22, % non-comm algebraic geometry
% %   14E15, % global theory and res'ns of sings
% %   14C40, % determinantal ideals
% 14F05, %Sheaves, derived categories of sheaves and related constructions
% 14M15; % Grassmannians in alg geom
%  Secondary:
% %   13D02, % syzygies and res'ns
% %   12G50, % rep. theory MCM modules
% %   16G20% repr'ns of quivers
% 16S38, % rings from non-comm alg geom
% 15A75% Grassmannians in linear algebra
% }


\begin{abstract}
  A construction due to Kn\"orrer shows that if $N$ is a maximal Cohen-Macaulay
  module over a hypersurface defined by $f+y^2$, then the first syzygy of $N/yN$
  decomposes as the direct sum of $N$ and its own first syzygy.  This was extended by
  Herzog-Popescu to hypersurfaces $f+y^n$, replacing $N/yN$ by $N/y^{n-1}N$.  We
  show, in the same setting as Herzog-Popescu, that the first syzygy of $N/y^{k}N$ is
  always an extension of $N$ by its first syzygy, and moreover that this extension
  has useful approximation properties. We give two applications. First, we construct
  a ring $\Lambda^\#$ over which every finitely generated module has an eventually
  $2$-periodic projective resolution, prompting us to call it a ``non-commutative
  hypersurface ring''.  Second, we give upper bounds on the dimension of the
  stable module category (a.k.a. the singularity category) of a hypersurface defined
  by a polynomial of the form $x_1^{a_1} + \dots + x_d^{a_d}$. 
\end{abstract}

\maketitle

\section{Introduction}\label{sect:intro}

Let $S$ be a regular local ring of dimension $d+1$, and let $f$ be a non-zero
non-unit of $S$. Let $R = S/(f)$ be the $d$-dimensional hypersurface ring.  Fix an
integer $n \geq 2$ and set $R^\# = S[y]/(f+y^n)$. We refer to $R^\#$ as the
\emph{($n$-fold) branched cover} of $R$.

We consider maximal Cohen-Macaulay (MCM) modules over $R$ and $R^\#$.  Our starting
point is the following theorem of Kn\"orrer and its generalization by Herzog-Popescu.
Write $\Omega$ for the syzygy operator.

\begin{theorem}[{\cite{Knorrer}}]
  \label{thm:Knorrer}
  Suppose that $S = k[\![x_0, \dots, x_d]\!]$ and the characteristic of $k$ is not
  equal to $2$. Assume $n=2$, so that $R^\# = S[y]/(f+y^2)$.  Then for any MCM
  $R^\#$-module $N$, we have
  \[
  \Omega_{R^\#}(N/yN) \cong N \oplus \Omega_{R^\#}(N)\,.
  \]
%  In particular $N \in \Sigma_1$.
\end{theorem}

\begin{theorem}[{\cite[Theorem 2.6]{Herzog-Popescu:1997}}]
  \label{thm:Herzog-Popescu}
  Suppose that $S = k[\![x_0, \dots, x_d]\!]$ and the characteristic of $k$ does not divide
   $n$. Let $n \geq 2$ be arbitrary.  Then for any MCM $R^\#$-module $N$,
  we have 
  \[
  \Omega_{R^\#}(N/y^{n-1}N) \cong N \oplus \Omega_{R^\#}(N)\,.
  \]
  % In particular
  % $N \in \Sigma_{n-1}$.
\end{theorem}

We remark that each of these results is a special case of the more general results
in the respective papers, the proofs of which do require the additional assumption
that $S$ be a power series ring over a field.  As a byproduct of our results below,
we will see that both results continue to hold true when $S$ is an arbitrary regular
local ring. 

In this note we consider the full subcategories of $R^\#$-modules of the form
$\Omega_{R^\#}(M)$, where $M$ is annihilated by $y^k$, for $k = 1, \dots,
n$. Specifically, set $A_k = S[y]/(f,y^k)$ for $1 \leq k \leq n$, and define
\[
\Sigma_k = \add\left\{\Omega_{R^\#}(M) \,\middle|\, M \text{ is a MCM
    $A_k$-module}\right\}\,.
\]

We are interested in the categories
$\Sigma_1 \subseteq \Sigma_2 \subseteq \cdots \subseteq \MCM(R^\#)$.  Observe that
the results of Kn\"orrer and Herzog-Popescu above imply $\Sigma_{1} = \MCM(R^\#)$
when $n=2$, and $\Sigma_{n-1}=\MCM(R^\#)$ for arbitrary $n$, respectively.

Our main result is the following.

{\def\thetheorem{A}
\begin{theorem}\label{thm:A}
  The subcategories $\Sigma_1, \dots, \Sigma_n$ are functorially finite in
  $\MCM(R^\#)$, that is, every MCM $R^\#$-module $N$ admits both left and right
  approximations by modules in $\Sigma_k$ for each $k$.  Namely, there are short
  exact sequences
  \[
  0 \to \Omega_{R^\#}(N) \to \Omega_{R^\#}(N/y^kN) \to N \to 0
  \]
  and
  \[
  0 \to N \to \Omega_{R^\#}(\Omega_{R^\#}(N)/y^k\Omega_{R^\#}(N)) \to
  \Omega_{R^\#}(N) \to 0
  \]
  which are right and left $\Sigma_k$-approximations of $N$, respectively.
\end{theorem}
}

We give two applications of Theorem~\ref{thm:A}.  

First we extend a result of Iyama, Leuschke, and Quarles to the effect that the
endomorphism ring of a representation generator for the MCM modules over a simple
(ADE) hypersurface singularity has finite global dimension.  When
$f \in k[\![x_0, \dots, x_d]\!]$ defines a simple singularity of dimension
$d \geq 1$, the polynomial $f+y^n$ does not in general define a simple singularity
for $n \geq 3$, so that in particular the branched cover does not have a
representation generator. However we show that the $n$-fold branched cover admits a
MCM module $M$ such that $\Lambda = \End_{R^\#}(M)$ behaves like a ``non-commutative
hypersurface ring,'' in the sense that $\Lambda$ is Iwanaga-Gorenstein and every
finitely generated $\Lambda$-module has a projective resolution which is eventually
periodic of period at most $2$.

Second we give upper bounds on the dimension, in the sense of Rouquier, of the stable
category of MCM modules over a complete hypersurface ring defined by
a polynomial of the form $x_0^{a_0} + \cdots + x_d^{a_d}$.  These bounds sharpen
the general bounds of Ballard-Favero-Katzarkov~\cite{Ballard-Favero-Katzarkov:2012}
for complete isolated hypersurface singularities.

\begin{notation*}
  Unless otherwise specified, all modules are left modules.
  
  For a ring $\Gamma$ and a left $\Gamma$-module $X$, we write $\Omega_\Gamma(X)$ for
  an arbitrary first syzygy of $X$. This is uniquely defined up to projective
  summands.  If $\Gamma$ is commutative local, $\Omega_\Gamma(X)$ is unique up to
  isomorphism.
\end{notation*}

\section{Approximating MCM modules over the branched cover}
\label{sect:approx}

We make essential use of the following general lemma. 

\begin{lemma}
  \label{lem:ses}
  Let $\Gamma$ be a ring and $X$ a $\Gamma$-module. Let $z \in \Gamma$ be a
  non-zerodivisor on $X$. Then there is a short exact sequence
  \begin{equation}\label{eq:gen-ses}
    0 \to \Omega_\Gamma(X) \to \Omega_\Gamma(X/zX) \to X \to 0
  \end{equation}
  of $\Gamma$-modules.
\end{lemma}

\begin{proof}
  Let $0 \to \Omega_\Gamma(X) \to F \to X \to 0$ be a free presentation of
  $X$. Multiplication by $z$ on $X$ induces a pullback diagram 
  \[
  \xymatrix{
            &    & 0 \ar[d]   & 0 \ar[d] \\
    0 \ar[r] & \Omega_\Gamma(X) \ar[r] \ar@{=}[d]  & P \ar[r] \ar[d] & X \ar[r]
    \ar[d]^{z} & 0 \\
    0 \ar[r] & \Omega_\Gamma(X) \ar[r]  & F \ar[r] \ar[d] & X \ar[r] \ar[d] &
    0\\
            &    & X/zX \ar@{=}[r] \ar[d]  & X/zX \ar[d] \\
            &    & 0          & 0
  }
  \]
  the middle column of which shows that $P \cong \Omega_\Gamma(X/zX)$.
\end{proof}

\begin{lemma}
  \label{lem:speed-kills}
  Let $Q$ be a commutative Noetherian ring, $X$ a finitely generated $Q$-module, and
  $z$ a non-zerodivisor on $X$. The following conditions are equivalent.
  \begin{enumerate}[\quad(i)]
  \item $\Omega_Q (X/zX) \cong X \oplus \Omega_Q(X)$;
  \item $z$ annihilates $\Ext_Q^1(X,\Omega_Q(X))$;
  \item $z$ annihilates $\Ext_Q^1(X,Y)$ for every finitely generated $Q$-module $Y$.
  \end{enumerate}
\end{lemma}

\begin{proof}
  Indeed, the short exact sequence~\eqref{eq:gen-ses} is the image in
  $\Ext_Q^1(X,\Omega_Q(X))$ of $z\cdot \id_X \in \End_Q(X)$. Therefore $z$ kills
  $\Ext_Q^1(X,\Omega_Q(X))$ if and only if~\eqref{eq:gen-ses} splits, which by
  Miyata's theorem~\cite{Miyata} happens if and only if
  $\Omega_Q(X/zX) \cong X \oplus \Omega_Q(X)$.  The equivalence of the second and
  third conditions is well known.
\end{proof}

Now we consider the $n$-fold branched cover $R^\#$ of a hypersurface ring $R$.  As in
the Introduction, let $S$ be a regular local ring of dimension $d+1$, let $f$ be a
non-zero non-unit of $S$, and let $R = S/(f)$ be the $d$-dimensional hypersurface
ring.  Fix $n \geq 2$ and set $R^\# = S[y]/(f+y^n)$. Then Lemma~\ref{lem:ses} yields
the following.

\begin{proposition}
  \label{prop:ses-mf}
  Let $N$ be a MCM $R^\#$-module. For each $k \geq 1$, we have:
  \begin{enumerate}[\quad(i)]
  \item\label{item:ses} There are short exact sequences
    \begin{equation}\label{eq:approx}
      0 \to \Omega_{R^\#}(N) \to \Omega_{R^\#}(N/y^kN) \to N \to 0
    \end{equation}
    and
    \begin{equation}
      \label{eq:otherapprox}
      0 \to N \to \Omega_{R^\#}(\Omega_{R^\#}(N)/y^k\Omega_{R^\#}(N)) \to
      \Omega_{R^\#}(N) \to 0\,.
    \end{equation}
  % \item\label{item:mf} The pair  
  %   \[
  %   \left(
  %     \begin{bmatrix} \psi & y^k I \\ 0 & \phi\end{bmatrix}\,,\ 
  %     \begin{bmatrix} \phi & -y^k I \\ 0 & \psi\end{bmatrix} 
  %   \right)
  %   \]
  %   is a matrix factorization of $f+y^n$ with cokernel isomorphic to
  %   $\Omega_{R^\#}(N/y^k N)$.
  \item \label{item:speed-kills} We have
    $\Omega_{R^\#}(N/y^k N) \cong N \oplus \Omega_{R^\#}(N)$ if and only if $y^k$
    annihilates $\Ext_{R^\#}^1(N, \Omega_{R^\#}(N))$.
  \end{enumerate}
\end{proposition}

\begin{proof}
  Existence of \eqref{eq:approx} and the second statement both follow immediately
  from Lemma~\ref{lem:ses}.  For \eqref{eq:otherapprox} it suffices to observe that
  since $R^\#$ is a hypersurface, if $N$ has no free direct summands, $N$ is
  isomorphic to its own second syzygy. If on the other hand $N$ is free, then both
  exact sequences degenerate to the isomorphism $R^\# \cong \Omega_{R^\#}(R^\#/(y^k))$.
\end{proof}

We emphasize that $f \neq 0$ is required for this result, since we need $y$ to be a
non-zerodivisor on $R^\#$. 

\begin{remark}
  \label{rmk:Ragnar} We observe that the results of Kn\"orrer and Herzog-Popescu from
  the Introduction hold true in our level of generality. The proof of~\cite[Theorem
  2.6]{Herzog-Popescu:1997} uses the \emph{Noether different} $\mathcal {N}_A^B$ of a
  ring homomorphism $A \to B$.  Let $\mu \colon B \otimes_A B \to B$ denote the
  multiplication map, and define
  \[
  \mathcal {N}_A^B = \mu\left(\Ann_{B\otimes_A B}( \ker \mu)\right)\,.
  \]
  The two results we need about the Noether different are contained
  in~\cite{Buchweitz:1987} in the generality required. 

  First,~\cite[Theorem 7.8.3]{Buchweitz:1987} implies that if $A \to B$ is a
  homomorphism of commutative Noetherian rings of finite Krull dimension, with $A$
  regular and $B$ Gorenstein, and such that $B$ is a finitely generated free
  $A$-module, then $\mathcal {N}_A^B$ annihilates all $\Ext$-modules $\Ext_B^1(X,Y)$
  with $X$, $Y$ finitely generated $B$-modules.

  The second result is contained in~\cite[Example 7.8.5]{Buchweitz:1987}. If $A$ is
  an arbitrary commutative ring and $B = A[y]/(g(y))$ for some polynomial $g(y) \in
  A[y]$, then the (formal) derivative $g'(y)$ is contained in $\mathcal {N}_A^B$.
\end{remark}

\begin{proposition}
  \label{prop:newHP}
  Let $S$ be an arbitrary regular local ring, $f \in S$ a non-zero non-unit,
  $R = S/(f)$, and $R^\# = S[y]/(f+y^n)$ for some $n \geq 1$. Assume that $n$ is a
  unit in $S$. Then for any MCM $R^\#$-module $N$, we have
  \[
  \Omega_{R^\#}(N/y^{n-1} N) \cong
  N \oplus \Omega_{R^\#}(N)\,.
  \]
\end{proposition}

\begin{proof}
  By Proposition~\ref{prop:ses-mf}(\ref{item:speed-kills}), it is enough to show
  that $y^{n-1}$ annihilates $\Ext_{R^\#}^1(N, \Omega_{R^\#}(N))$.  Since
  $\frac{\partial} {\partial y}(f+y^n) = n y^{n-1}$ and $n$ is invertible, this
  follows from Remark~\ref{rmk:Ragnar}.
\end{proof}

Next we want to consider the relationship between the short exact
sequence~\eqref{eq:approx} and the module $N$.  We recall here some definitions from
``relative homological algebra''.

\begin{definition}
  \label{def:approx}
  Let $\cala$ be an abelian category, $\calx \subseteq \cala$ a full subcategory, and
  $M$ an object of $\cala$.
  \begin{enumerate}[\quad(i)]
  \item A \emph{right $\calx$-approximation} (or \emph{$\calx$-precover}) of $M$ is a
    morphism $q \colon X \to M$, with $X \in \calx$, such that any $f \colon X' \to
    M$ with $X' \in \calx$ factors through $q$. Equivalently, the induced
    homomorphism $\Hom_\cala (X',X) \to \Hom_\cala(X',M)$ is surjective.
  \item A \emph{left $\calx$-approximation} (or \emph{$\calx$-preenvelope}) is a
    morphism $j \colon M \to X$ with $X \in \calx$ such that any $g \colon M \to X'$
    with $X' \in \calx$ factors through $j$.  Equivalently the induced homomorphism
    $\Hom_\cala(X,X') \to \Hom_\cala(M,X')$ is surjective. 
  \item A right $\calx$-approximation $q \colon X \to M$ is \emph{minimal} (or a
    \emph{$\calx$-cover}) if every endomorphism $\phi \in \End_\cala(X)$ with $\phi q = q$ is an
    isomorphism.  Dually, a left $\calx$-approximation $j \colon M \to X$ is
    \emph{minimal} (a \emph{$\calx$-envelope}) if every endomorphism $\phi
    \in \End_\cala(X)$ with $j \phi = j$ is an isomorphism. 
  \item We say that $\calx$ is \emph{contravariantly finite} (resp.\
    \emph{covariantly finite}) in $\cala$ if every object $M$ of $\cala$ has a right
    (resp.\ left) $\calx$-approximation. 
  \item A sequence $A \stackrel{f}{\to} B \stackrel{g}{\to} C$ is
    \emph{$\calx$-exact} (resp. \emph{$\calx$-coexact}) if the induced sequence
    $\Hom_\cala(X,A) \stackrel{f_*}{\to} \Hom_\cala(X,B) \stackrel{g_*}{\to}
    \Hom_\cala(X,C)$ (resp.
    $\Hom_\cala(C,X) \stackrel{g^*}{\to} \Hom_\cala(B,X) \stackrel{f^*}{\to}
    \Hom_\cala(A,X)$) is exact for all $X \in \calx$.
  \end{enumerate}
\end{definition}

Approximations need not exist, and even when they do, minimal approximations need not
exist.

We routinely abuse language by referring to the object $X$ as an approximation of $M$
(on the appropriate side). Furthermore, when $\cala$ is a module category and a right
(resp.\ left) $\calx$-approximation happens to be surjective (resp.\ injective) we
will refer to the whole short exact sequence $0 \to \ker q \to X \xto{q} M \to 0$
(resp.\ $0 \to M \xto{j}  X \to \cok j \to 0$) as the approximation.

The next result contains Theorem~\ref{thm:A} from the Introduction.  Recall the
definition of $\Sigma_k$:
\[
\Sigma_k = \add\left\{\Omega_{R^\#}(M) \,\middle|\, M \text{ is a MCM
    $A_k$-module}\right\}\,.
\]

\begin{theorem}\label{thm:approx}
  Let $N$ be a MCM $R^\#$-module.  The short exact sequence \eqref{eq:approx} is a
  right $\Sigma_k$-approximation of $N$, and \eqref{eq:otherapprox} is a left
  $\Sigma_k$-approximation of $N$.
\end{theorem}

\begin{proof}
  To show that \eqref{eq:approx} is a $\Sigma_k$-approximation, it suffices to show
  that an arbitrary $R^\#$-homomorphism $\sigma \colon M \to N$, for
  $M \in \Sigma_k$, factors through $\Omega_{R^\#}(N/y^k N)$.  Equivalently, $\sigma$
  maps to zero in $\Ext_{R^\#}^1(M, \Omega_{R^\#}(N))$ in the long exact sequence
  \[
  \Hom_{R^\#}(M,\Omega_{R^\#}(N/y^k N)) \to \Hom_{R^\#}(M,N) \to
  \Ext_{R^\#}^1(M,\Omega_{R^\#}(N))
  \]
  obtained by applying $\Hom_{R^\#}(M,-)$ to \eqref{eq:approx}.

  Recall that \eqref{eq:approx} is the image of the short exact sequence
  $\chi = (0 \to \Omega_{R^\#}(N) \to F \to N \to 0) \in
  \Ext_{R^\#}^1(N,\Omega_{R^\#}(N))$ under multiplication by $y^k$.  It therefore
  suffices to show that $y^k\chi$ maps to zero in
  $\Ext_{R^\#}^1(M, \Omega_{R^\#}(N))$.  Since $M \in \Sigma_k$, we can write
  $M\oplus M' \cong \Omega_{R^\#}(X)$ for an $R^\#$-module $X$ annihilated by $y^k$
  and some $R^\#$-module $M'$. Then
  \[
  \begin{split}
    \Ext_{R^\#}^1(M\oplus M', \Omega_{R^\#}(N))
    &= \Ext_{R^\#}^1(\Omega_{R^\#}(X), \Omega_{R^\#}(N))\\
    &\cong \Ext_{R^\#}^2(X, \Omega_{R^\#}(N))
  \end{split}
  \]
  is annihilated by $y^k$ since $X$ is, which proves the claim.

  The statement about left approximation follows similarly, using the fact that
  \[
  \begin{split}
    \Ext_{R^\#}^1(N,M\oplus M')
    &= \Ext_{R^\#}^1(N, \Omega_{R^\#}(X))\\
    &\cong \Ext_{R^\#}^2(\Omega_{R^\#}(N), \Omega_{R^\#}(X))\\
    &\cong \Ext_{R^\#}^1(\Omega_{R^\#}(N), X)\
  \end{split}
  \]
  is also killed by $y^k$.
\end{proof}

Standard arguments about approximations now yield the following.

\begin{corollary}\label{cor:splitApprox} For a MCM $R^\#$-module $N$ and $k \geq 1$, the following are equivalent:
\begin{enumerate}
\item $N \in \Sigma_k$;
\item The sequence~\eqref{eq:approx} splits;
\item The sequence~\eqref{eq:otherapprox} splits;
\item $\Omega_{R^\#}(N) \in \Sigma_k$;
\item $y^k$ annihilates $\Ext^1_{R^\#}(N,\Omega_{R^\#}(N))$.
\end{enumerate}
\qed
\end{corollary}

\begin{remark}\label{rmk:relativeHomology} In \cite{AusSol:1993} Auslander and
  Solberg develop relative homological algebra from the perspective of sub-functors
  of $\Ext^1(-,-)$.  For any $k \geq 1$, $F_k(-,-) := y^k \Ext^1_{R^\#}(-,-)$ is such
  a sub-functor on $\MCM(R^\#)$.  The corresponding category of relatively
  $F_k$-projectives consists of all $N$ such that
  $F_k(N,-) = y^k \Ext^1_{R^\#}(N,-) = 0$, which coincides with $\Sigma_k$.
  Likewise, the category of relatively $F_k$-injectives, consisting of all $N$ such
  that $F_k(-,N) = y^k \Ext^1_{R^\#}(-,N) = 0$, also coincides with $\Sigma_k$, as
  $\Sigma_k$ is closed under syzygies and $\Omega^2_{R^\#}$ is isomorphic to the
  identity.  The fact that $\Sigma_k$ is functorially finite translates to the
  existence of enough relative $F_k$-projectives (and injectives).  Moreover, as the
  relative projectives and injectives coincide, we see that for each $k\geq 1$
  $\MCM(R^\#)$ has a (relative) Frobenius category structure with $\Sigma_k$ as the
  (relative) projectives.
\end{remark}

\begin{prop}
  Assume in addition that $S$ is complete and that $N$ is indecomposable.  Then
  \eqref{eq:approx} is a minimal right approximation of $N$ as long as it is not
  split.
\end{prop}

\begin{proof}
  Since MCM $R^\#$-modules satisfy the Krull-Remak-Schmidt uniqueness property for
  direct-sum decompositions, a right approximation $q \colon X \to M$ is non-minimal
  if and only if it vanishes on a direct summand of $X$.  When $N$ is indecomposable,
  the kernel of the approximation $\Omega_{R^{\#}}(N)$ is indecomposable as well, and
  it cannot contain a direct summand of $X$ other than itself, in which case the
  sequence splits.
\end{proof}

It does not seem possible in general to identify the first $k$ so that
\eqref{eq:approx} splits.  By Proposition~\ref{prop:ses-mf} and the long exact
sequence in $\Ext$, it coincides with the least $k$ so that multiplication $y^k$ on
$N$ factors through a free $R^\#$-module. Equivalently, multiplication by $y^k$ is in
the image of the natural map $N^* \otimes_{R^\#} N \to \End_{R^\#}(N)$.  In the
special case $n=2$, it is known (see for example~\cite[Proposition
8.30]{Leuschke-Wiegand:BOOK}) that for an indecomposable MCM $R$-module $M$,
$\Omega_{R^\#}(M)$ is indecomposable if and only if $M \cong \Omega_R(M)$.  As far as
we are aware, no criterion of this form is known for $n \geq 3$.

\section{Branched covers of simple singularities} \label{sect:simple}

In this section we consider the case where $R = S/(f)$ is a hypersurface of
\emph{finite CM representation type} (of dimension $d\geq 1$). If $S$ is a power
series ring over an algebraically closed field of characteristic zero, this is
equivalent to $R$ being a \emph{simple singularity,} that is,
$R \cong k[\![x,y,z_2, \dots, z_d]\!]/(f)$, where $f$ is one of the following
polynomials:
\begin{enumerate}
  \setlength\itemindent{.2\linewidth}
\item[$(A_n)$:]  \qquad$x^2 + y^{n+1}+z_2^2+\cdots+z_d^2$ , \qquad $n \geq 1$
\item[$(D_n)$:]  \qquad$x^2y + y^{n-1}+z_2^2+\cdots+z_d^2$ , \qquad $n \geq 4$
\item[$(E_6)$:]  \qquad$x^3 + y^4+z_2^2+\cdots+z_d^2$
\item[$(E_7)$:]  \qquad$x^3 + xy^3+z_2^2+\cdots+z_d^2$
\item[$(E_8)$:]  \qquad$x^3 + y^5+z_2^2+\cdots+z_d^2$
\end{enumerate} 
We point out that each of these is an iterated double ($2$-fold) branched cover of
the one-dimensional simple singularity; Kn\"orrer's Theorem~\ref{thm:Knorrer} allows
one to show that these hypersurface rings have finite CM representation type by
reducing to the $1$-dimensional case.  For $n \geq 3$, adding a summand of the form
$z^n$ to a simple singularity gives a non-simple singularity (except in type $A_1$
and some isolated examples for small $n$).

Since $R$ has finite CM representation type, the category
$\Sigma_1 \subset \MCM (R^\#)$ has only finitely many indecomposable objects up to
isomorphism.  Let $M$ be a representation generator for $\MCM R$, that is, $M$
satisfies $\add M = \MCM R$, so that in particular $\tilde{M} := \Omega_{R^\#}(M)$ is
a representation generator for $\Sigma_1$.  We note that $\tilde{M}$ has a direct
summand isomorphic to $R^\#$.

We set $\Lambda_0 = \End_R(M)$ and $\Lambda = \End_{R^\#}(\tilde{M})$.  Recall the
following result due independently to Iyama~\cite{Iyama:Aus-corr},
Leuschke~\cite{Leuschke:finrepdim}, and Quarles~\cite{Quarles:thesis}.

\begin{prop}\label{prop:gldimfinite}
  The endomorphism ring $\Lambda_0$ has global dimension at most $\max\{2, d\}$,
  with equality holding if $d \geq 2$. \qed
\end{prop}

The main result of this section is that $\Lambda$ can be thought of as a
``non-commutative hypersurface'', in that every finitely generated $\Lambda$-module
has a projective resolution which is eventually periodic of period at most $2$.
Furthermore $\Lambda$ is Iwanaga-Gorenstein, and we identify the relative
Gorenstein-projectives.

\begin{definition}\label{def:completeres}
  Let $A$ be a ring and $M$ a left $A$-module.  A \emph{complete resolution} of $M$
  is an exact sequence of projective $A$-modules
  \[
  \PP_{\bullet} = \cdots \to P_{1} \stackrel{d_1}{\to} P_0 \stackrel{d_0}{\to} P_{-1}
  \to \cdots
  \]
  such that $\Hom_A(\PP_{\bullet},A)$ is exact and $M \cong \im (d_0)$.  A module $M$
  admitting a complete resolution is said to be \emph{Gorenstein projective} or
  \emph{totally reflexive}.  We will write $\GP(A)$ for the full subcategory of
  $\lmod{A}$ consisting of Gorenstein projective modules.
\end{definition}

We regard $\tilde{M}$ as a left $R^\#$-module and a right $\Lambda$-module, so that
we have functors $\Hom_{R^\#}(\tilde{M}, - ) \colon \lMod{R^\#} \to \lMod{\Lambda}$
and $\Hom_{R^\#}(-,\tilde{M}) \colon \lMod{R^\#} \to \rMod{\Lambda}$ which induce
equivalences from $\add{\tilde{M}}$ to the subcategories of finitely generated
projective left (respectively, right) $\Lambda$-modules.

\begin{lemma}\label{lemma:homfunctors}
  Restricted to $\MCM(R^\#)$, we have a commutative diagram of functors up to natural
  isomorphism.
  \[
    \xymatrixcolsep{5.0pc} \xymatrix{\MCM(R^\#) \ar[r]^{\Hom_{R^\#}(\tilde{M},-)} 
    \ar[dr]_{\Hom_{R^\#}(-,\tilde{M})\ \ } & \Lambda\mbox{-}\mod 
    \ar[d]^{\Hom_{\Lambda}(-,\Lambda)} \\ & \mod\mbox{-}\Lambda}
  \]
\end{lemma}

\begin{proof}
  By Theorem~\ref{thm:approx}, any MCM $R^\#$-module $N$ has a right
  $\Sigma_1$-approximation
  \begin{equation}\label{eq:rtApprox}
    0 \to \Omega_{R^\#} (N) \to T \to N \to 0\,,
  \end{equation}
  where in fact $T = \Omega_{R^\#}(N/yN)$. Furthermore $\Omega_{R^\#}(N)$ has a right
  $\Sigma_1$-approximation 
  \begin{equation}\label{eq:ltApprox}
    0 \to N \to \Omega_{R^\#} (T) \to \Omega_{R^\#} (N)  \to 0\,.
  \end{equation}
  Since these sequences are $\Sigma_1$-exact and $\Sigma_1$-coexact, we obtain exact
  sequences of $\Lambda$-modules
  \[
    0 \to \Hom_{R^\#}(\tilde{M},N) \to \Hom_{R^\#}(\tilde{M},\Omega_{R^\#}(T)) \to
    \Hom_{R^\#}(\tilde{M}, T) \to \Hom_{R^\#}(\tilde{M},N) \to 0
  \]
  and 
  \[
    0 \to \Hom_{R^\#}(N,\tilde{M}) \to \Hom_{R^\#}(T,\tilde{M}) \to
    \Hom_{R^\#}(\Omega_{R^\#}(T),\tilde{M}) \to \Hom_{R^\#}(N,\tilde{M}) \to 0\,,
  \]
  which yield projective presentations for the $\Lambda$-modules
  $\Hom_{R^\#}(\tilde{M},N)$ and $\Hom_{R^\#}(N,\tilde{M})$ respectively.  Thus these
  modules are finitely presented over $\Lambda$.  We now apply
  $\Hom_{\Lambda}(-,\Lambda)$ to the first of these and compare with the second to
  obtain an exact commutative diagram
  \[\small
    \xymatrix{ 0 \ar[r] & \Hom_{\Lambda}(\Hom_{R^\#}(\tilde{M},N),\Lambda) \ar[r] & \Hom_{\Lambda}(\Hom_{R^\#}(\tilde{M},T),\Lambda) \ar[r] & \Hom_{\Lambda}(\Hom_{R^\#}(\tilde{M},\Omega_{R^\#}(T)),\Lambda) \\
      0 \ar[r] & \Hom_{R^\#}(N,\tilde{M}) \ar@{-->}[u] \ar[r] &
      \Hom_{R^\#}(T,\tilde{M}) \ar[u]^{\cong} \ar[r] &
      \Hom_{R^\#}(\Omega_{R^\#}(T),\tilde{M}) \ar[u]^{\cong}}
  \]
  Here the vertical maps are induced by the functor $\Hom_{R^\#}(\tilde{M},-)$, and
  the two on the right are isomorphisms since
  $T, \Omega_{R^\#}(T) \in \add{\tilde{M}}$.  It follows that the induced map on the
  left is also an isomorphism.
\end{proof}

\begin{prop}\label{prop:completeres}
  Let $Z$ be a left (resp.\ right) $\Lambda$-module of the form
  $\Hom_{R^\#}( \tilde{M}, N)$ (resp. $\Hom_{R^\#}(N,\tilde{M})$), where $N$ is a MCM
  $R^\#$-module. Then $Z$ has a complete $\Lambda$-resolution which is periodic of
  period at most $2$.
\end{prop}

\begin{proof}
  As in the preceding proof we have exact sequences \eqref{eq:rtApprox} and
  \eqref{eq:ltApprox}, which we can splice together to form a doubly-infinite
  $2$-periodic exact sequence of $R^\#$-modules
 \begin{equation}\label{eq:completeSigmaRes}
  \TT\colon \cdots \to \Omega_{R^\#} (T) \to T \to \Omega_{R^\#} (T) \to T \to \cdots\,.
  \end{equation}
  Since $\tilde{M}$ is in $\Sigma_1$, the induced sequences
  $\Hom_{R^\#}(\tilde{M},\TT)$ and $\Hom_{R^\#}(\TT,\tilde{M})$ are still exact, and
  since both $T$ and $\Omega_{R^\#} (T)$ are in $\Sigma_1 = \add{\tilde{M}}$, each
  module in the sequence is a projective $\Lambda$-module. Thus
  $\Hom_{R^\#}(\tilde{M},\TT)$ (respectively, $\Hom_{R^{\#}}(\TT,\tilde{M})$) is a
  projective resolution (and co-resolution) of $Z = \Hom_{R^\#}(\tilde{M}, N)$
  (respectively, of $Z' = \Hom_{R^\#}(N,\tilde{M})$).

  It remains to show that $\Hom_{R^\#}(\tilde{M},\TT)$ and
  $\Hom_{R^\#}(\TT,\tilde{M})$ are complete resolutions, that is, remain exact upon
  dualizing into $\Lambda$.  However, by Lemma~\ref{lemma:homfunctors} we know that
  these resolutions are dual to each other under $\Hom_{\Lambda}(-,\Lambda)$, and the
  result follows since they are both exact.
  %By the Yoneda  Lemma, however, we have a canonical isomorphism
  %\[ 
  %\Hom_{\Lambda}(\Hom_{R^\#}(\Omega_{R^\#}(M),\TT),
  %\Hom_{R^\#}(\Omega_{R^\#}(M),\Omega_{R^\#}(M)))
  %\cong
  %\Hom_{R^\#}(\TT, \Omega_{R^\#}(M))\,,
  %\]
  %which is exact by Theorem~\ref{thm:approx} again.
\end{proof}

\begin{remark}
  In fact the resolution we build is something stronger than periodic of period $2$
  -- it is built out of a single map $f \colon \Omega (T) \to T$ and its syzygy
  $\Omega (f) \colon T \to \Omega(T)$.
\end{remark}

\begin{lemma}
  \label{lem:MCMsyz}
  Set $m = \max\{2,d+1\}$.  Let $X$ be an arbitrary finitely generated left (resp.\
  right) $\Lambda$-module.  Then there is a MCM $R^\#$-module $N$ such that
  $\Omega_\Lambda^m(X) \cong \Hom_{R^\#}(\tilde{M},N)$ (resp.
  $\Omega_{\Lambda}^m(X) \cong \Hom_{R^\#}(N,\tilde{M})$).
\end{lemma}

\begin{proof}
  First assume that $X$ is a left $\Lambda$-module.  Set $Z = \Omega_\Lambda^m(X)$, and let 
  \[
  0 \to Z \to P_{m-1} \xto{\partial_{m-1}} P_{m-2} \to \cdots \to P_1
  \xto{\partial_1} P_0 \to X \to 0
  \]
  be the first $m-1$ steps of a resolution of $X$ by finitely generated projective
  $\Lambda$-modules. By the Yoneda Lemma, each $\partial_{j} \colon P_j \to P_{j-1}$
  is of the form $\Hom_{R^\#}( \tilde{M}, f_j)$ for some $R^\#$-linear
  $f_j\colon M_j \to M_{j-1}$, where each $M_j \in \add \tilde{M}$.  (Notice
  that since $m \geq 2$, there is indeed at least one $\partial_j$.)  We thus obtain
  a sequence of $R^\#$-modules and homomorphisms
  \begin{equation}\label{eq:Yoneda}
    0 \to \ker f_{m-1} \to M_{m-1} \xto{f_{m-1}} M_{m-2} \to \cdots \to M_1 \xto{f_1}
    M_0 \,.
  \end{equation}
  Since $\tilde{M}$ has an $R^\#$-free summand and the result of applying
  $\Hom_{R^\#}( \tilde{M},-)$ to \eqref{eq:Yoneda} is exact, in fact
  \eqref{eq:Yoneda} must be exact.  In particular $N := \ker f_{m-1}$  has depth at
  least $m \geq d+1$, and is hence a MCM $R^\#$-module. Furthermore, we have
  $\Hom_{R^\#}(\tilde{M},N) \cong Z$ by left-exactness of $\Hom$.

  Similarly, if $X$ is a right $\Lambda$-module, we can take a projective resolution
  $\PP \cong \Hom_{R^\#}(\MM,\tilde{M})$ of $X$, where $\MM$ denotes a sequence
  $M_0 \stackrel{f_1}{\to} M_1 \to \cdots \to M_{m-2} \stackrel{f_{m-1}}{\to} M_{m-1}
  \to \cdots$ in $\add{\tilde{M}}$. Since $R^\#$ is a summand of $\tilde{M}$ and
  $\PP$ is exact, we have an exact sequence
  \[
    0 \to K \to M_{m-1}^* \stackrel{f_{m-1}^*}{\to} M_{m-2}^* \to \cdots \to M_1^*
    \stackrel{f_1^*}{\to} M_0^*\,,
  \] 
  where we set $K := \ker(f_{m-1}^*)$, writing $(-)^*$ for the exact duality
  $\Hom_{R^\#}(-,R^\#)$. Since each $M_i^*$ is a MCM $R^{\#}$-module, $K$ has depth
  at least $m \geq d+1$, and is hence a MCM $R^\#$-module.  Taking the $R^\#$-dual
  now shows that $K^* \cong \cok(f_{m-1})$, and the left-exactness of
  $\Hom_{R^\#}(-,\tilde{M})$ yields that
  $\Hom_{R^\#}(K^*,\tilde{M}) \cong \ker \Hom_{R^\#}(f_{m-1},\tilde{M}) \cong
  \Omega_{\Lambda}^{m}(X)$.
\end{proof}

\begin{theorem}\label{thm:hypersurface}
  Let $R = S/(f)$ be a hypersurface of finite CM representation type, with
  representation generator $M$. Let $R^\#  = S[y]/(f+y^n)$ for some $n \geq 1$, and
  set $\Lambda = \End_{R^\#}(\Omega_{R^\#}(M))$. Then every finitely 
generated left (resp.\ right) $\Lambda$-module has a projective resolution which is
  eventually periodic of period at most $2$.
\end{theorem}

\begin{proof}
  Let $X$ be a finitely generated left (resp.\ right) $\Lambda$-module. By
  Lemma~\ref{lem:MCMsyz}, $\Omega_\Lambda^m(X)$ is of the form
  $\Hom_{R^\#}(\tilde{M},N)$ (resp. $\Hom_{R^\#}(N,\tilde{M})$) for some MCM
  $R^\#$-module $N$, where $m = \max\{2, \dim R + 1\}$. It then follows from
  Proposition~\ref{prop:completeres} that $\Omega_\Lambda^m(X)$ has a $2$-periodic
  complete resolution.  Splicing the left-hand half of this complete resolution
  together with the first $m-1$ steps of the projective resolution of $X$, we obtain
  an eventually periodic resolution.
\end{proof}

\begin{remark}\label{rem:notreally}
  While we think of the ring $\Lambda^\#$ as a ``non-commutative hypersurface'' in a
  homological sense, on the basis of Theorem~\ref{thm:hypersurface}, we should
  emphasize that as far as we know there does not exist a ring $\Gamma$ of finite
  global dimension such that $\Lambda^\# \cong \Gamma/(g)$ for some element $g$.
\end{remark}

\begin{remark}\label{rem:manyobjects}
  It seems likely that some version of many of the results in this section hold more
  generally (that is, without the assumption that $R$ is a simple singularity) if one
  uses the formalism of ``rings with several objects'' as in, for
  example,~\cite{Holm:2015}.  We don't pursue this direction.
\end{remark}

Recall that a noetherian ring $A$ is said to be Iwanaga-Gorenstein of dimension at
most $d$ if $\id_A A \leq d$ and $\id A_A \leq d$.

\begin{proposition}\label{prop:injectivedimension} 
  The ring $\Lambda = \End_{R^\#}(\Omega_{R^\#}(M))$ constructed above is
  Iwanaga-Gorenstein of dimension at most $m = \max\{2,d+1\}$.
\end{proposition}

\begin{proof}
  Set $m = \max\{2,d+1\}$ and observe that for any finitely generated left (resp.\ right)
  $\Lambda$-module $X$, we have
  \[
  \Ext^{m+1}_{\Lambda}(X,\Lambda) \cong \Ext^1_{\Lambda}(\Omega^m_{\Lambda}(X),\Lambda) \cong
  0\,,
  \]
  since $\Omega_\Lambda^m(X)$ is Gorenstein projective by Lemma~\ref{lem:MCMsyz}
  and Proposition~\ref{prop:completeres}, and
  Gorenstein projectives are Ext-orthogonal to $\Lambda$ by definition.  It follows
  that $\id {}_{\Lambda} \Lambda \leq m$ and $\id \Lambda_{\Lambda} \leq m$.
\end{proof}

\begin{theorem}\label{thm:Gprojectives} The functor $F := \Hom_{R^\#}(\tilde{M},-)$
  induces an equivalence of categories from $\MCM(R^\#)$ to $\GP(\Lambda)$.  In
  particular, a left $\Lambda$-module $X$ is Gorenstein projective if and only if
  $X \cong \Hom_{R^\#}(\tilde{M},N)$ for a MCM $R^\#$-module $N$.
\end{theorem}

\begin{proof} We begin with the second claim, which establishes denseness of $F$.  By
  Proposition~\ref{prop:completeres} we know that every $\Lambda$-module of the form
  $\Hom_{R^\#}(\tilde{M},N)$ with $N$ in $\MCM(R^\#)$ is Gorenstein projective.
  Conversely, suppose that $X$ belongs to $\GP(\Lambda)$.  By Lemma~\ref{lem:MCMsyz}
  $\Omega_{\Lambda}^m(X) \cong \Hom_{R^\#}(\tilde{M},N)$ for some $N$ in
  $\MCM(R^\#)$, and thus by Proposition~\ref{prop:completeres},
  $\Omega_{\Lambda}^m(X)$ has a complete resolution of the form
  $\Hom_{R^\#}(\tilde{M},\TT)$ for an exact sequence $\TT$ as in
  \eqref{eq:completeSigmaRes}. Here, as there, $m = \max\{2, d+1\}$.  Comparing this
  complete resolution to that of $X$ we see that, up to projective direct summands,
  $X \cong \Hom_{R^\#}(\tilde{M},N)$ or
  $X \cong \Hom_{R^\#}(\tilde{M},\Omega_{R^\#}N)$ depending on whether $m$ is even or
  odd, respectively.

  The functor $F$ is easily seen to be faithful since $R^\#$ is a direct summand of
  $\tilde{M}$.  To see that it is full, consider a map $g \colon F(N) \to F(N')$ for
  $N, N'$ in $\MCM(R^\#)$.  Then $g$ can be lifted to a map of projective
  presentations as in the diagram,
  \[
    \xymatrix{F(M_1) \ar[r]^{F(f_1)} \ar[d]^{g_1} & F(M_0) \ar[r]^{F(f_0)} \ar[d]^{g_0}
      & F(N) \ar[d]^g \ar[r] & 0 \\ F(M'_1) \ar[r]^{F(f'_1)} & F(M'_0) \ar[r]^{F(f'_0)}
      & F(N') \ar[r] & 0}
  \]
  with $M_i, M'_i \in \add{\tilde{M}}$.  Since $F$ induces an equivalence from
  $\add{\tilde{M}}$ to $\add(\Lambda)$, we can write $g_i = F(h_i)$ for suitable maps
  $h_i$ (for $i=1,2$), and we obtain a commutative exact diagram in $\MCM(R^\#)$
  \[
    \xymatrix{M_1 \ar[r]^{f_1} \ar[d]^{h_1} & M_0 \ar[r]^{f_0} \ar[d]^{h_0} & N
      \ar@{-->}[d]^h \ar[r] & 0 \\ M'_1 \ar[r]^{f'_1} & M'_0 \ar[r]^{f'_0} & N'
      \ar[r] & 0}
  \]
  yielding a map $h \colon N \to N'$.  It then follows that $g = F(h)$.
\end{proof}

\begin{remark}\label{rmk:FrobeniusCategories}  
  The preceding Proposition and Theorem can also be deduced from recent work of
  Iyama, Kalck, Wemyss and Yang \cite{IKWY:2015}.  To see this, recall from Remark~\ref{rmk:relativeHomology} 
  that the category $\MCM(R^\#)$ has a
  (relative) Frobenius category structure obtained by taking
  $\Sigma_1 =\add{\tilde{M}}$ as the subcategory of (relative) projective objects,
  and $\Sigma_1$-exact sequences as the exact sequences.  
  %For instance, Theorem~\ref{thm:approx} assures that there are enough (relative) projectives and injectives for this exact structure.  Moreover, in the context of relative  homological algebra developed by Auslander and Solberg \cite{AusSol:1993}, it appears that this relative Frobenius structure corresponds to the subfunctor of  $\Ext^1_{R^\#}(-,-)$ consisting of extensions in the image of multiplication by $y$.
   Since the functor categories $\rmod{\MCM(R^\#)}$ and
  $\rmod{\MCM(R^\#)^{\op}}$ have global dimension at most $m = \max\{2,d+1\}$
  by~\cite{Holm:2015}, Theorem 2.8 of \cite{IKWY:2015} applies.  Part (1) of that
  theorem corresponds to the Iwanaga-Gorenstein property of our
  $\Lambda = \End_{R^\#}(\tilde{M})$, while part (2) gives the equivalence between
  $\MCM(R^\#)$ and $\GP(\Lambda)$.  Of course, it also follows from this realization
  of a Frobenius structure, that we have an equivalence of triangulated categories
  induced by
  \[
    \Hom_{R^\#}(\tilde{M},-) \colon  \MCM(R^\#)/[\Sigma_1] \approx
  \underline{\GP}(\Lambda)\,,
  \]
  where $[\Sigma_1]$ denotes the ideal of morphisms that
  factor through an object in the subcategory $\Sigma_1$ and we write
  $\underline{\GP}(\Lambda)$ for the stable category of Gorenstein projective
  $\Lambda$-modules.
\end{remark}

\section{Generation of MCM module categories}
\label{sect:generate}

When $R$ is a Gorenstein local ring, $\MCM(R)$ is a Frobenius category, and thus the
stable category $\uMCM(R)$, whose objects are the same as $\MCM(R)$ and whose
$\Hom$-sets are obtained by killing those morphisms factoring through a free
$R$-module, is triangulated with suspension given by the co-syzygy functor
$\Omega^{-1}$~\cite{Buchweitz:1987}.  This stable category is also
equivalent~\cite{Orlov:DerCatsCohTriang} to the \emph{singularity category}
$D_{\text{sg}}(R)$, that is, the Verdier quotient of the bounded derived category
$D^b(\mod R)$ by the subcategory of perfect complexes. In this section we investigate
the dimension of this triangulated category, in the sense of Rouquier, when $R$ is an
isolated hypersurface singularity.  In this setting, Ballard, Favero and Katzarkov
have found a general upper bound for the dimension of $\uMCM(R)$, showing in
particular that it is always finite \cite{Ballard-Favero-Katzarkov:2012}.  We focus
on a special class of hypersurfaces, where we improve upon this upper bound.

We begin by reviewing the definition of the dimension of a triangulated category
introduced by Rouquier~{\cite{Rouquier:2008}}.  Let $\calt$ be a triangulated
category and $\cali$ a full subcategory of $\calt$.  (For ease of exposition, 
we will assume all full subcategories are closed under isomorphisms.)  
We let $\gen{\cali}$ denote the
smallest full subcategory of $\calt$ that contains $\cali$ and is closed under
isomorphisms, direct summands, finite direct sums and shifts.  For two full
subcategories $\cali_1$ and $\cali_2$ we also write $\cali_1 * \cali_2$ for the full
subcategory of all $Y$ for which $\calt$ contains a distinguished triangle
$X_1 \to Y \to X_2 \to X_1[1]$ with $X_i \in \cali_i$, and we further set
$\cali_1 \diamond \cali_2 = \gen{\cali_1 * \cali_2}$.  Observe that if $\calt = \uMCM(R)$, then $Y \in \cali_1 * \cali_2$
if and only if there is a short exact sequence $0 \to X_1 \to Y \oplus F \to X_2 \to 0$ in $\MCM(R)$ with $X_i \in \cali_i$ for $i=1,2$ and $F$ free.  Following the conventions of
\cite{Ballard-Favero-Katzarkov:2012}, we inductively define
$\gen{\cali}_0 = \gen{\cali}$ and
$\gen{\cali}_{n} = \gen{\cali}_{n-1} \diamond \gen{\cali}$ for $n \geq 1$.

\begin{definition}[Rouquier]
  \label{def:DimTriCat} The \emph{dimension} of $\calt$ is
  the smallest integer $n$ such that $\calt = \gen{X}_n$ for some object $X$ in
  $\calt$, or else $\infty$ if no such $n$ exists.
\end{definition}

% Ballard, Favero and Katzarkov~{\cite{Ballard-Favero-Katzarkov:2012}} have shown
% that the stable category of MCM modules over any isolated hypersurface singularity
% has finite dimension, and they give an upper bound for this dimension.
If $R = k[\![x_0,\ldots,x_d]\!]/(f)$ is a complete isolated hypersurface singularity
over an algebraically closed field $k$ of characteristic zero, Ballard, Favero and
Katzarkov bound the dimension of $\uMCM(R)$ in terms of the Tjurina algebra
$k[\![x_0,\ldots,x_d]\!]/\Delta_f$, where $\Delta_f$ is the ideal in
$k[\![x_0,\ldots,x_d]\!]$ generated by the partial derivatives of $f$ with respect to
the $x_i$.  Recall that this algebra is Artinian and local.

\begin{theorem}[{\cite{Ballard-Favero-Katzarkov:2012}}]
  \label{thm:HypersurfaceCatDim}
  Let $R$ be a $d$-dimensional isolated hypersurface singularity as above, with $\ell$
  denoting the Loewy length of the Tjurina algebra of $R$.  Then
  $\uMCM(R) = \gen{\Omega^d(k)}_{2\ell-1}$, and in particular
  $\dim(\uMCM(R)) \leq 2 \ell - 1$.
\end{theorem}

%This is not exactly the same statement they give, but it should suffice here.

In particular, if $R$ is defined by the polynomial
$f = x_0^{a_0} + \cdots + x_d^{a_d}$, where each $a_i$ is at least $2$ and not
divisible by $\charac(k)$, then the Tjurina algebra of $R$ is isomorphic to
$k[x_0,\ldots,x_d]/(x_0^{a_0-1}, \ldots, x_d^{a_d-1})$, and it is easy to see that
its Loewy length is
\[
\ell = \sum_{i=0}^d (a_i - 2) + 1 = \sum_{i=0}^d a_i - 2d - 1\,.
\]
In this case, we have $\dim(\underline{\MCM}(R)) \leq 2\sum_{i=0}^d a_i - 4d - 3$.

%\todo{Remarks about using extensions/short exact sequences to compute $\cali_1 * \cali_2$ for subcategories of the stable category.--Added one sentence in 2nd paragraph.}

We now return to the context explored in the previous section.  Thus assume that $R$
is an isolated hypersurface singularity given by a non-zero power series $f$ in
$x_0, \ldots, x_d$ and set $R^\# = k[\![x_0,\ldots,x_d,y]\!]/(f+y^n)$.  For
$1 \leq i \leq n$, we set
\[
  \Sigma_i = \add \left\{\Omega_{R^\#}(M) \,\middle|\, M \in
    \MCM(R^\#/(y^i))\right\}\,,
\] 
and continue to write $\Sigma_i$ for its image in the stable category.  Note that $\Sigma_i$ is closed under isomorphisms, finite direct sums and direct summands by definition, and is closed under shifts by Corollary~\ref{cor:splitApprox}.

The following lemma, which is a slight generalization of \cite[Lemma
2.4]{Herzog-Popescu:1997}, will be used repeatedly in what follows.  As the lemma can
be stated over any ring $\Gamma$, we recall that syzygies are only defined up to
projective summands.  Thus for each left $\Gamma$-module $M$ we fix a projective resolution $(F_{i}^M,\partial_{i}^M)_{i\geq 0}$ of $M$ and define
$\Omega^i_\Gamma(M) = \ker \partial_{i-1}^M$.  We say that two $\Gamma$-modules $X$ and $Y$ are stably isomorphic if $X \oplus P \cong Y \oplus P'$ for some projective $\Gamma$-modules $P$ and $P'$.

\begin{lemma}\label{lemma:doubleSyzygy}
  Let $\Gamma$ be any ring, let $I$ be a proper two-sided ideal of $\Gamma$ such that
  $\bar{\Gamma} = \Gamma/I$ has finite projective dimension $r \geq 1$ as a left
  $\Gamma$-module.  For any left $\bar{\Gamma}$-module $X$ we have
\[
\Omega^{r+1}_\Gamma(X) \oplus P \cong \Omega^{r}_\Gamma(\Omega_{\bar{\Gamma}}(X)) \oplus P'
\]
for projectives $P$ and $P'$.  More generally, for any $m \geq 1$ there exist
projectives $P, P'$ such that
\[
\Omega^{r+m}_\Gamma(X) \oplus P \cong \Omega^{r}_\Gamma(\Omega^{m}_{\bar{\Gamma}}(X)) \oplus P'\,.
\]
\end{lemma}

\begin{proof} Fix a projective resolution
  $\cdots \to F_1 \stackrel{\partial_1}{\to} F_0 \stackrel{\partial_0}{\to} X \to 0$ so that
  $\Omega^i_\Gamma(X) = \ker \partial_{i-1}$ for each $i \geq 1$.
  Tensoring down to $\bar{\Gamma}$, we obtain an exact commutative diagram of
  $\Gamma$-modules
  \[
    \xymatrix{
      & 0 \ar[d] & 0 \ar[d] \\
      & IF_0 \ar[d] \ar@{=}[r] & IF_0 \ar[d] \\
      0 \ar[r] & \Omega_\Gamma X \ar[d] \ar[r] & F_0 \ar[d] \ar[r] & X \ar@{=}[d] \ar[r] & 0 \\
      0 \ar[r] & K \ar[r] \ar[d] & F_0/IF_0 \ar[r] \ar[d] & X \ar[r] & 0 \\
      & 0  & 0
    } 
  \]
  where $K \oplus Q_1 \cong \Omega_{\bar{\Gamma}}X \oplus Q_2$ for projective
  $\bar{\Gamma}$-modules $Q_1, Q_2$ by Schanuel's Lemma.  Since $\pd ({}_{\Gamma}\bar{\Gamma}) = r$, $\Omega^{r}_{\Gamma}(Q_i)$ is a projective $\Gamma$-module for each $i$.  Thus, $\Omega^{r}_{\Gamma}K$ and $\Omega^{r}_\Gamma(\Omega_{\bar{\Gamma}} X)$ will be stably isomorphic over $\Gamma$.
  
  We then have a second diagram
  as below.
  \[
    \xymatrix{
      & & & 0 \ar[d] \\
      & 0 \ar[d] & 0 \ar[d] & IF_0 \ar[d] \\
      0 \ar[r] & \Omega^2_\Gamma X \ar[d] \ar[r] & F_1 \ar@{=}[d] \ar[r] & \Omega_\Gamma X \ar[d] \ar[r] & 0 \\
      0 \ar[r] & L \ar[d] \ar[r] & F_1 \ar[d] \ar[r] & K \ar[r] \ar[d] & 0 \\
      & IF_0 \ar[d] & 0  & 0 \\
      & 0
    }
  \]
  Considering the vertical sequence on the left, we can use the horseshoe lemma to obtain a projective resolution of $L$ and a short exact sequence of syzygies
 $$0 \to \Omega^{r+1}_\Gamma X \to \Omega^{r-1}_\Gamma L \oplus P \to \Omega^{r-1}_\Gamma (IF_0) \oplus P' \to 0$$
for projective $\Gamma$-modules $P$ and $P'$.  Since $IF_0 \in \add(I)$ and ${}_\Gamma I$ has projective dimension $r-1$, $\Omega^{r-1}_\Gamma(IF_0)$ is projective and this sequence splits.  It follows that $\Omega^{r+1}_\Gamma X$ is stably isomorphic to $\Omega^{r-1}_{\Gamma} L$.   Since $L$ is stably isomorphic to $\Omega_\Gamma (K)$, we see that $\Omega^{r+1}_\Gamma X$ is stably isomorphic to $\Omega^{r}_\Gamma K$ and thus to $\Omega^{r}_\Gamma(\Omega_{\bar{\Gamma}}X)$.
  
  The second claim now follows via induction on $m\geq 1$.
\end{proof}

 % old version - less general
 % $$\xymatrix{& 0 \ar[d] & 0 \ar[d] \\
 % & zF_0 \ar[d] \ar@{=}[r] & zF_0 \ar[d] \\
 % 0 \ar[r] & \Omega_R X \ar[d] \ar[r] & F_0 \ar[d] \ar[r] & X \ar@{=}[d] \ar[r] & 0 \\
 % 0 \ar[r] & \Omega_{\bar{R}}X \oplus \bar{F} \ar[r] \ar[d] & F_0/zF_0 \ar[r] \ar[d] & X \ar[r] & 0 \\
 % & 0  & 0 } $$
 % where $F$ is a free $R$-module and $\bar{F} = F/zF$.  We also have
 % $$\xymatrix{& & & 0 \ar[d] \\
 % & 0 \ar[d] & 0 \ar[d] & zF_0 \ar[d] \\
 % 0 \ar[r] & \Omega^2_R X \ar[d] \ar[r] & F_1 \ar@{=}[d] \ar[r] & \Omega_R X \ar[d] \ar[r] & 0 \\
 % 0 \ar[r] & \Omega_R(\Omega_{\bar{R}}X) \oplus F' \ar[d] \ar[r] & F_1 \ar[d] \ar[r] & \Omega_{\bar{R}}X \oplus \bar{F} \ar[r] \ar[d] & 0 \\
 % & zF_0 \ar[d] & 0  & 0 \\
 % & 0}$$
 % Since $zF_0$ is free the vertical sequence on the left splits and we have the desired result.

 % Continuing the notation of the Lemma, induction on $r$ now shows $\Omega^r_{R}(X) \cong \Omega_R(\Omega^{r-1}_{\bar{R}}(X))$,
 % up to free summands, for all $r \geq 2$.

In particular, when the ideal $I$ is generated by a central regular sequence of length $r$, we have $\pd ({}_\Gamma \Gamma/I) = r$ and thus the preceding lemma applies in this situation.   Furthermore, for such ideals $I$ in a commutative ring $\Gamma$, the Eagon-Northcott resolution shows that $\pd ({}_\Gamma \Gamma/I^j) = r$ for each $j \geq 1$ \cite{Eagon-Northcott}, and thus the lemma also applies for the ideal $I^j$, as will be needed a bit later on.

\begin{prop}\label{prop:SigmastarSigma}
  We have $\Sigma_k = \Sigma_{k-j} \diamond \Sigma_{j}$ for every
  $1 \leq j < k \leq n$.  In particular,
  $\MCM(R^\#) = \Sigma_{n-1} = \gen{\Sigma_1}_{n-2}$.
\end{prop}

\begin{proof}
  First let $M \in \Sigma_k$, say $M \oplus M' = \Omega_{R^\#}(X)$ with $X$ an
  $R^\#$-module annihilated by $y^k$. Then we have a short exact sequence
  \begin{equation}\label{eq:nonMCM}
    0 \to y^j X \to X \to X/y^j X \to 0\,,
  \end{equation}
  in which the leftmost, resp.\ rightmost, term is annihilated by $y^{k-j}$, resp.\
  $y^j$. Observe, however, that neither of these need be MCM over $R^\#/(y^{k-j})$,
  resp.\ $R^\#/(y^{j})$.  Fix an odd integer $m > d$, and apply $\Omega_{R^\#}^m$ to
  \eqref{eq:nonMCM}, obtaining
  \begin{equation}\label{eq:nowMCM}
    0 \to \Omega_{R^\#}^m(y^j X) \to \Omega_{R^\#}^m(X) \oplus F \to
    \Omega_{R^\#}^m(X/y^j X) \to 0 
  \end{equation}
  for some free $R^\#$-module $F$.  Since $m$ is odd, the middle term is isomorphic
  to $\Omega_{R^\#}^1(X) \oplus F \cong M \oplus M' \oplus F$. By
  Lemma~\ref{lemma:doubleSyzygy}, since powers of $y$ are regular in $R^\#$, we may rewrite the outer terms as the first
  syzygies, over $R^\#$, of $\Omega_{R^\#/(y^{k-j})}^{m-1}(y^j X)$ and
  $\Omega_{R^\#/(y^{j})}^{m-1}(X/y^j X)$, respectively. For $m$ large enough, these
  are both MCM over the appropriate quotient of $R^\#$, and it follows that
  $\Omega^m_{R^\#}(y^j X) \in \Sigma_{k-j}$ and $\Omega^m_{R^\#}(X/y^j X) \in
  \Sigma_{j}$. Hence $M \in \Sigma_{k-j} \diamond \Sigma_{j}$.

  For the other containment, it suffices to complete a diagram of the form
  \[
    \xymatrix{
      & 0  \ar[d]    & & 0  \ar[d]\\
      0 \ar[r] & \Omega_{R^\#}(X_1) \ar[r] \ar[d] & M\oplus M' \ar[r] &
      \Omega_{R^\#}(X_2) \ar[r]
      \ar[d] & 0 \\
      & F_1 \ar[d] & & F_2 \ar[d] \\
      & X_1 \ar[d] & & X_2 \ar[d] \\
      & 0 & & 0 }
  \]
  where $y^{k-j}X_1 = y^j X_2 = 0$ and $F_1$, $F_2$ are free $R^{\#}$-modules.  Apply
  $\Hom_{R^\#}(-,F_1)$ to the top row, obtaining 
  \[
  \cdots \to \Hom_{R^\#}(M\oplus M',F_1) \to \Hom_{R^\#}(\Omega_{R^\#}(X_1),F_1) \to
  \Ext_{R^\#}^1(\Omega_{R^\#}(X_2),F_1) \,.
  \]
  Since $\Omega_{R^\#}(X_2)$ is MCM and $F_1$ is free,
  $\Ext_{R^\#}^1(\Omega_{R^\#}(X_2),F_1)=0$. This implies that the map
  $\Omega_{R^\#}(X_1) \to F_1$ extends to a map $M \oplus M' \to F_1$, and we therefore obtain
  a commutative diagram of exact sequences
  \[
    \xymatrix{ 0 \ar[r] & \Omega_{R^\#}(X_1) \ar[r] \ar[d] & M \oplus M'\ar[r] \ar[d]
      & \Omega_{R^\#}(X_2) \ar[r]
      \ar[d] & 0 \\
      0 \ar[r] & F_1 \ar[r] & F_1 \oplus F_2 \ar[r] & F_2 \ar[r] & 0 \,.  }
  \]
  Set $X = \cok (M\oplus M' \to F_1\oplus F_2)$. Then of course
  $M \oplus M' = \Omega_{R^\#}(X)$ and we have a short exact sequence
  $0 \to X_1 \to X \to X_2 \to 0$. Since $y^{k-j}X_1=0$ and $y^jX_2 =0$, we have
  $y^k X=0$, as required.
\end{proof}

We now extend this result to iterated branched covers, as considered by O'Carroll and
Popescu in \cite{OCarroll-Popescu:2000JPAA}.  
 Let $S$ be a $d+1$-dimensional regular local
ring, and set $R = S/(f)$ and
\[
R_r = S[\![y_1,\ldots,y_r]\!]/(f+y_1^{a_1} + \cdots y_r^{a_r})
\] 
for positive integers $a_i \geq 2$.  Let $I_r$ be the ideal
$(y_1, \ldots, y_r)R_r$, and define
\[
  \Sigma^r_j := \add \{\Omega^r_{R_r}(X) \ |\ X \in \MCM({R_r}/I_r^j)\} \,,.
\]
Note that $y_1, \dots, y_r$ is a regular sequence on $R_r$, so that $R_r/I_r^j$ is a
CM ring of dimension $d$ for any $j$ (by, for example, Eagon-Northcott~\cite[Theorem
2]{Eagon-Northcott}). It follows that any $X \in \MCM({R_r}/I_r^j)$ has depth $d$ as
an $R_r$-module, and thus that $\Sigma_j^r \subseteq \MCM(R_r)$.

The argument outlined in Remark~\ref{rmk:Ragnar} has been extended by
Takahashi~\cite{Takahashi:2010} to recover the following result of O'Carroll and
Popescu~\cite{OCarroll-Popescu:2000JPAA}: 
\begin{prop}\label{prop:takahashi}
  Assume that each $a_i$ is invertible in $S$.  Then for any MCM $R_r$-module
  $N$, we have
  \begin{equation}
    \Omega_{R_r}^r (N/(y_1^{a_1-1}, \dots, y_r^{a_r-1})N) 
    \cong \bigoplus_{j=0}^r \Omega_{R_r}^j (N)^{r \choose j}\,.\tag*{\qed} 
  \end{equation}
\end{prop}
The necessary ingredients in proving Proposition~\ref{prop:takahashi} are $(i)$ to
observe that the Jacobian ideal $(y_1^{a_1-1}, \dots, y_r^{a_r-1})$ annihilates
$\Ext_{R_r}^1(N, \Omega_{R_r}(N))$, and $(ii)$ to induct over the number
$r$ of new variables by applying Lemma~\ref{lemma:doubleSyzygy}. We omit the details.

We set $m := \sum_{i=1}^r(a_i-2) + 1$, so that
$I_r^m \subseteq (y_1^{a_1-1}, \ldots, y_r^{a_r-1})$. By
Proposition~\ref{prop:takahashi}, every MCM $R_r$-module $N$ is a direct
summand of $\Omega_{R_r}^r(N/(y_1^{a_1-1}, \ldots, y_r^{a_r-1})N)$.  Since
$N/(y_1^{a_1-1}, \ldots, y_r^{a_r-1})N$ is a MCM $R_r/I_r^m$-module, this yields
$\Sigma^r_m = \MCM(R_r)$.

\begin{theorem}\label{thm:newSigmaStarSigma} For each $1 \leq j < k$ we have
  $\Sigma^r_k = \Sigma^r_{k-j} \diamond \Sigma^r_j = \gen{\Sigma^r_1}_{k-1}$.  In
  particular, $\MCM(R_r) = \Sigma^r_{m} = \gen{\Sigma^r_1}_{m-1}$.
\end{theorem}

\begin{proof}
  Step 1.  We first show the inclusions
  $\Sigma^r_k \subseteq \Sigma^r_{k-j} \diamond \Sigma^r_j \subseteq
  \gen{\Sigma^r_1}_{k-1}$.  The argument is similar to the first part of the proof of Proposition~\ref{prop:SigmastarSigma}.  Let $M \in \Sigma^r_k$, say
  $M\oplus M' = \Omega^r_{R_r}(X)$ with $X$ an $R_r$-module annihilated by
  $I_r^k$. Then, for any even integer $n>d$, the short exact sequence $0 \to I_r^j X \to X \to X/I_r^j X \to 0$
  induces a short exact sequence
  $0 \to \Omega^{r+n}_{R_r}(I_r^j X) \to M\oplus M' \oplus F \to \Omega^{r+n}_{R_r}(X/I_r^j
  X) \to 0$ for some free module $F$. By Lemma~\ref{lemma:doubleSyzygy}, $\Omega^{r+n}_{R_r}(X/I_r^j X)$ is stably isomorphic to $\Omega^r_{R_r}(\Omega^n_{R_r/I_r^j} (X/I_r^j X))$ with $\Omega^n_{R_r/I_r^j} (X/I_r^j X)$ MCM over $R_r/I_r^j$.  Similarly,  $\Omega^{r+n}_{R_r}(I_r^j X)$ is stably isomorphic to $\Omega^r_{R_r}(\Omega^n_{R_r/I_r^{k-j}} (I_r^j X))$ with $\Omega^n_{R_r/I_r^{k-j}} (I_r^j X)$ MCM over $R_r/I_r^{k-j}.$  In particular,
  $M \in \Sigma^r_{k-j} \diamond \Sigma^r_{j}$.  The second inclusion now follows by
  definition of $\gen{\Sigma^r_1}_{k-1}$ and induction on $k$.

  \vspace{3mm} Step 2.  We show that
  $\Omega_{R_r}(\Sigma^{r-1}_i) \subseteq \Sigma^r_i$ for each $r \geq 2$ and
  $i \geq 1$.  Let $M \in \Sigma^{r-1}_i$, so that
  $M \oplus M' \cong \Omega^{r-1}_{R_{r-1}}(X)$ for some $M' \in \MCM(R_{r-1})$ and
  $X \in \MCM(R_{r-1}/I_{r-1}^i)$.  Then
  \[
    \Omega_{R_r}M \oplus \Omega_{R_r}M' \cong \Omega_{R_r}(\Omega^{r-1}_{R_{r-1}}(X))
    \cong  \Omega^r_{R_r}(X)\,,
  \]
  where the last isomorphism holds up to free summands by
  Lemma~\ref{lemma:doubleSyzygy}.  Thus $\Omega_{R_r}M$ belongs to $\Sigma^r_i$ since
  $X$ may also be viewed as a MCM module over $R_{r}/I_r^i$.

  \vspace{3mm} Step 3.  We show that $\Sigma_1^r = \Omega_{R_r}(\Sigma_1^{r-1})$ for
  each $r \geq 2$.  Let $M \in \Sigma_1^r$.  Then $M \oplus M' \cong \Omega^r_{R_r}X$
  for a MCM $R_r/I_r$-module $X$.  But $R_r/I_r \cong R \cong R_{r-1}/I_{r-1}$, so
  $X$ is an $R_{r-1}$-module and $\Omega^{r-1}_{R_{r-1}}X \in \Sigma^{r-1}_1$.  Thus,
  up to free summands,
  $M \oplus M' \cong \Omega_{R_r}(\Omega^{r-1}_{R_{r-1}}X) \in
  \Omega_{R_r}(\Sigma^{r-1}_1)$ as required.

  \vspace{3mm} Step 4.  For each $j \geq 1$, consider the following two statements:
  \begin{equation}\label{eqn:S1}
    \forall\ r \geq 2 \ \ \Sigma^r_j = \Omega_{R_r}(\Sigma^{r-1}_j)\,,
  \end{equation}
  and
  \begin{equation}\label{eqn:S2}
    \forall\ r \geq 1 \ \ \add(\Sigma^r_1 * \Sigma^r_j) = \Sigma^r_{j+1}\,.
  \end{equation}
  For any fixed $j \geq 1$, we show that \eqref{eqn:S1} implies \eqref{eqn:S2}.  To
  see this, we apply induction on $r$, noting that the $r=1$ case of \eqref{eqn:S2}
  is handled by Proposition~\ref{prop:SigmastarSigma}.  By Step 1, it suffices to
  show $\add(\Sigma^r_1 * \Sigma^r_j) \subseteq \Sigma^r_{j+1}$.  Let $r \geq 2$.
  For any $M \in \add(\Sigma^r_1 * \Sigma^r_j)$, we have an extension
  $0 \to N \to M \oplus M' \to N' \to 0$ for
  $N \in \Sigma^r_1= \Omega_{R_r}(\Sigma^{r-1}_1)$ and
  $N' \in \Sigma^r_j = \Omega_{R_r}(\Sigma^{r-1}_j)$.  Now, as in the second half of
  the proof of Proposition~\ref{prop:SigmastarSigma}, we obtain
  $M \in \Omega_{R_r}(\Sigma^{r-1}_1 * \Sigma^{r-1}_j) =
  \Omega_{R_r}(\Sigma^{r-1}_{j+1}) \subseteq \Sigma^r_{j+1}$, where we have used the
  inductive hypothesis and the result of Step 2.

  \vspace{3mm} Step 5.  We now prove \eqref{eqn:S1}, and hence also \eqref{eqn:S2},
  for all $j \geq 1$ by induction on $j$.  For $j=1$, \eqref{eqn:S1} was established
  in Step 3.  Now, assume $\eqref{eqn:S1}$ holds for some $j \geq 1$, and let
  $M \in \Sigma^r_{j+1} = \add(\Sigma^r_1 * \Sigma^r_j)$.  The argument in the
  previous step shows that $M \in \Omega_{R_r}(\Sigma^{r-1}_{j+1})$, thus
  establishing (\ref{eqn:S1}) for $j+1$ with the help of Step 2.

  \vspace{3mm} Step 6.  We finally show that $\Sigma^r_j = \gen{\Sigma^r_1}_{j-1}$
  for all $r, j \geq 1$.  It then follows that
  $\Sigma^r_k = \Sigma^r_{k-j} \diamond \Sigma^r_j$ for all $1 \leq j < k$.  By Step
  1, it suffices to show that $\gen{\Sigma^r_1}_{j-1} \subseteq \Sigma^r_j$.  For
  $j=1$, this amounts to saying that $\Sigma^r_1$ is closed under syzygies.  For
  $r=1$, we know that $\Sigma^r_j$ is closed under syzygies from
  Corollary~\ref{cor:splitApprox}.  By induction on $r$, using (\ref{eqn:S1}) and
  Lemma~\ref{lemma:doubleSyzygy}, we see that all $\Sigma^r_j$ are in fact closed
  under syzygies.  Now using induction on $j$, we have
  \[
    \gen{\Sigma^r_1}_{j} = \gen{\Sigma^r_1 * \gen{\Sigma^r_1}_{j-1}} = \gen{\Sigma_1^r *
      \Sigma^r_j} = \gen{\Sigma^r_{j+1}} = \Sigma^r_{j+1}\,.
  \]
\end{proof}

\begin{corollary}\label{cor:findim}
  With notation as above, $\dim (\MCM(R_r)) \leq \sum_{i=1}^r(a_i-2)$ whenever $R$
  has finite CM-type.
\end{corollary}

\begin{remark}
  We emphasize again that the Corollary above requires $f \neq 0$. Indeed, it is
  known that the category of MCM modules over the $A_\infty$ singularity
  $k[\![x,y]\!]/(y^2)$ has dimension 1~\cite[Proposition
  3.7]{Dao-Takahashi:dimension}. 
\end{remark}

\begin{question}  We have shown that $\MCM(R_r) = \Sigma^r_{m} = \gen{\Sigma^r_1}_{m-1}$ and $\gen{\Sigma^r_1}_{m-2} = \Sigma^r_{m-1}$.  But we do not know if  $m-1$ is the smallest number of steps in which $\Sigma^r_1$ generates $\MCM(R_r)$, or equivalently whether $\Sigma^r_{m-1}$ must be a proper subset of $\Sigma^r_m = \MCM(R_r)$.  Properness here would follow from the existence of an MCM $R_r$-module $M$ for which $\Ext^1_{R_r}(M,-)$ is not annihilated by $\displaystyle y := \prod_{i=1}^r y_i^{a_i-2}$.  To see this, note that if such an $M$ is a direct summand of some $\Omega_{R_r}^r(X)$, then $\Ext^1_{R_r}(M,-)$ is isomorphic to a direct summand of $\Ext^{r+1}_{R_r}(X,-)$, and the latter must not be killed by $y$.  Since $y \in I_r^{m-1}$, $X$ is not annihilated by $I_r^{m-1}$ and thus $M \notin \Sigma^r_{m-1}$.
\end{question}

\section{Examples}
\label{sect:egs}

We illustrate the constructions in this paper in a couple of examples.  For
computational purposes it is convenient to work with matrix factorizations, on which
the necessary background can be found in Yoshino's book \cite{Yoshino:book}.  For a
power series $f$ contained in the maximal ideal of $S = k[\![\underline{x}]\!]$, we write $\MF_S(f)$
for the category of reduced matrix factorizations of $f$ over $R$ and recall that the
functor $\cok \colon \MF_R(f) \to \MCM(R/(f))$, sending a matrix factorization
$(\phi, \psi)$ to $\cok(\phi)$, induces equivalences of categories
$\MF_R(f)/[(1,f)] \approx \MCM(R/(f))$ and
$\MF_R(f)/[(1,f),(f,1)] \approx \uMCM(R/(f))$.  We also write
$\Omega(\phi,\psi) = (\psi,\phi)$.

We begin with a description of the functor
$\Omega_{R^\#}\colon \MCM(R) \rightarrow \MCM(R^\#)$ in terms of matrix
factorizations.  On this level, it turns out that this functor is a special case of
Yoshino's tensor product of matrix factorizations \cite{Yoshino:1998}.  Let
$R = k[\![x_0,\ldots,x_d]\!]$ and $R' = k[\![y_0,\ldots,y_{d'}]\!]$, and set
$S= k[\![x_0,\ldots,x_d,y_0,\ldots,y_{d'}]\!]$.  If $X = (\phi, \psi)$ and
$X' = (\phi',\psi')$ are matrix factorizations of $f \in R$ and $g \in R'$
respectively, of sizes $n$ and $m$, then Yoshino defines the matrix factorization
\[
  X \hat{\otimes} X' = \left( \begin{bmatrix} \phi \otimes I_m & I_n \otimes \phi' \\
      -I_n \otimes \psi' & \psi \otimes I_m \end{bmatrix}, \begin{bmatrix} \psi
      \otimes I_m & -I_n \otimes \phi' \\ I_n \otimes \psi' & \phi \otimes
      I_m \end{bmatrix}\right) \in \MF_S(f+g)\,.
\]
For a fixed $X'$, Yoshino shows that $-\hat{\otimes} X'\colon \MF_R(f) \to \MF_S(f+g)$ is
an exact functor that preserves trivial matrix factorizations, and behaves well with
respect to syzygies.  In particular,
$X \hat{\otimes} \Omega X' \cong \Omega(X \hat{\otimes} X') \cong \Omega X
\hat{\otimes} X'$.  In addition, if $X'$ is reduced, $-\hat{\otimes} X'$ is a faithful
functor (although, it is typically very far from being full).  We also have a
reduction functor $-\otimes_S S/(\underline{y})\colon \MF_S(f+g) \to \MF_R(f)$, which sends
$(\Phi, \Psi)$ to $(\Phi \otimes_S S/(\underline{y}), \Psi \otimes_S S/(\underline{y}))$.

We now specialize to the setting considered in this paper: $R' = k[\![y]\!], g = y^n$
and $Y = (y,y^{n-1})$.  Then the functor $-\hat{\otimes} Y$ sends an $m \times m$
matrix factorization $(\phi, \psi)$ of $f$ over $R$ to the $2m \times 2m$ matrix
factorization
\[
(\phi, \psi) \hat{\otimes} (y,y^{n-1}) = \left( \begin{bmatrix} \phi & yI_m \\
    -y^{n-1}I_m & \psi \end{bmatrix}, \begin{bmatrix} \psi & -yI_m \\ y^{n-1}I_m &
    \phi \end{bmatrix} \right)\,.
\]

\begin{prop}\label{prop:syzygyMF} Let $R = k[\![x_0,\ldots,x_d]\!]/(f)$ be an
  isolated hypersurface singularity and set
  $R^\# = k[\![x_0,\ldots,x_d,y]\!]/(f + y^n)$.  If $(\phi, \psi)$ is a reduced matrix
  factorization of $f$ corresponding to the MCM $R$-module $M = \cok (\phi, \psi)$,
  then we have $\Omega_{R^\#}(M) \cong \cok \Omega( (\phi,\psi) \hat{\otimes} Y))$.
\end{prop} 

\begin{proof}
The proof given in~\cite[Lemma 8.17]{Leuschke-Wiegand:BOOK} or \cite[Lemma
12.3]{Yoshino:book} in the case $n=2$ applies equally well for arbitrary $n$.
\end{proof}

\begin{example}
  Consider $R = k[\![x]\!]/(x^4)$ and $R^\# = k[\![x,y]\!]/(x^4+y^3)$, which is a
  simple curve singularity of type $\EE_6$.  Of course $R$ is a finite-dimensional
  $k$-algebra of finite representation type, and $\lmod{R}$ is easily pictured via
  its Auslander-Reiten quiver
\[
  \xymatrix{R/(x) \ar[r]<0.5ex> & R/(x^2) \ar[r]<0.5ex> \ar[l]<0.5ex> & R/(x^3)
    \ar[r]<0.5ex> \ar[l]<0.5ex> & R \ar[l]<0.5ex>}\,.
\]

For reference we also provide the Auslander-Reiten quiver of $R^\#$, following the
notation of \cite{Yoshino:book}, Chapter 9.
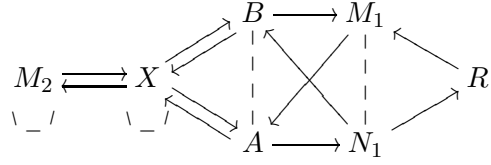
\[
  \xymatrixrowsep{0.75pc} \xymatrix{ & & B \ar[dl]<0.5ex> \ar[r] \ar@{--}[dd] & M_1
    \ar[ddl] \ar@{--}[dd] \\ M_2 \ar[r]<0.5ex> \ar@{--}@(dl,dr) & X \ar@{--}@(dl,dr)
    \ar[ur]<0.5ex> \ar[dr]<0.5ex> \ar[l]<0.5ex> & & & R \ar[ul] \\ & & A
    \ar[ul]<0.5ex> \ar[r] & N_1 \ar[uul] \ar[ur]}
\]
We have
\[
  \Omega_{R^\#}(R/(x)) \cong \cok((x^3,x) \hat{\otimes} (y,y^2)) = \cok
  \left( \begin{bmatrix} x^3 & -y \\ y^2 & x \end{bmatrix}, \begin{bmatrix} x & y \\
      -y^2 & x^3 \end{bmatrix} \right) \cong \cok(\psi_1, \phi_1) \cong N_1\,;
\]
 and
\[
\Omega_{R^\#}(R/(x^3)) \cong \cok((x,x^3) \hat{\otimes} (y,y^2)) \cong \Omega(N_1)
\cong M_1\,.
\]
Furthermore

\[
  \Omega_{R^\#}(R/(x^2)) \cong \cok ((x^2,x^2) \hat{\otimes} (y,y^2)) = \cok
  \left( \begin{bmatrix} x^2 & -y \\ y^2 & x^2 \end{bmatrix}, \begin{bmatrix} x^2 & y
      \\ -y^2 & x^2 \end{bmatrix} \right) \cong \cok(\psi_2, \phi_2) \cong N_2 \cong
  M_2\,.
\]

Thus $\Sigma_1 = \add(M_1 \oplus N_1 \oplus M_2 \oplus R^\#)$.  We compute the
minimal $\Sigma_1$-approximations of the remaining indecomposables in $\MCM(R^\#)$.
From the matrix factorizations for $A$, $B$ and $X$ given in \cite{Yoshino:book} (with $x$ and $y$ swapped), it
is easy to see that $A/yA \cong (R/(x^3))^2 \oplus R/(x^2)$,
$B/yB \cong (R/(x))^2 \oplus R/(x^2)$ and
$X/yX \cong (R/(x^2))^2 \oplus R/(x) \oplus R/(x^3)$.  Applying $\Omega_{R^\#}$ now
yields the middle terms in the following short exact sequences, which are left and
right $\Sigma_1$-approximations of the modules on the ends:
\begin{align*}
& 0 \to B \to M_1^2 \oplus M_2 \to A \to 0 & \\
& 0 \to A \to N_1^2 \oplus M_2 \to B \to 0 &\\
& 0 \to X \to M_1 \oplus N_1 \oplus M_2^2 \to X \to 0&
\end{align*}

In light of Proposition~\ref{prop:SigmastarSigma}, there must also be short exact
sequences with end terms in $\Sigma_1$ and each of $B, A$ and $X$ as direct summands
of the middle term.  In fact these extensions are evidenced by the following almost
split sequences

\begin{align*}
& 0 \to M_1 \to A \to N_1 \to 0 & \\
& 0 \to N_1 \to B \oplus R^\# \to M_1 \to 0 & \\
& 0 \to M_2 \to X \to M_2 \to 0. & 
\end{align*}

Set $\tilde M = M_1 \oplus N_1 \oplus M_2 \oplus R^\#$.  The ring
$\Lambda = \End_{R^\#}(\tilde M)$ is isomorphic to the corner ring of the Auslander
algebra of $\MCM(R^\#)$ corresponding to the 4 vertices $M_1, N_1, M_2$ and $R^\#$ in
the AR-quiver above.  To describe this algebra explicitly in terms of quivers and
relations, we identify $R^\#$ with $k[\![t^3,t^4]\!]$, and note that all the
indecomposables in $\Sigma_1$ can be represented as fractional ideals (for example,
see \cite{Yoshino:book}):
\[
M_1 \cong (t^3,t^8), \ \ N_1 \cong (t^3, t^4), \ \ M_2 \cong (t^6,t^8)\,.
\]
The irreducible morphisms between the indecomposables in $\Sigma_1$ are easily seen to correspond to the following paths in the
AR-quiver of $\MCM(R^\#)$.  Moreover each is realized by multiplication by a suitable power of $t$.
\begin{eqnarray*} 
	t^5 & : & N_1 \to B \to X \to M_2, \\
 	t^4 & : & N_1 \to B \to X \to B \to M_1, \\
  	1 & : & N_1 \to R^\#, \\
   	t^3 & : & R^\# \to M_1, \\
   	t^3 & : &M_1 \to A \to X \to M_2,  \\
   	1 & : &M_1 \to A \to X \to A \to N_1, \\
    	t^{-2} & : & M_2 \to X \to A \to N_1, \\
    	1 & : &M_2 \to X \to B \to M_1, \\
     	t^2 & : & M_2 \to X \to A \to X \to M_2.
\end{eqnarray*}
It follows that $\Lambda$ can be described as a factor of the completed path algebra of the
following quiver\footnote{with the convention that arrows are composed left to right}

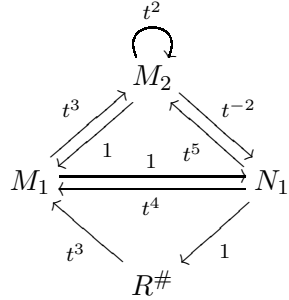
\[
  \xymatrix{& M_2 \ar@(ul,ur)^{t^2} \ar[dr]<0.5ex>^{t^{-2}} \ar[dl]<0.5ex>^1 \\ M_1
    \ar[ur]<0.5ex>^{t^3} \ar[rr]<0.5ex>^1 & & N_1 \ar[ul]<0.5ex>^{t^5}
    \ar[ll]<0.5ex>^{t^4} \ar[dl]^{1} \\ & R^\# \ar[ul]^{t^3}}
\]

While we don't list all the relations here, we note that they can be easily
identified from the above quiver, as they correspond to parallel paths that compose
to identical powers of $t$.  For instance, among the minimal relations we find the
difference between the paths $M_2 \stackrel{1}{\to} M_1 \stackrel{1}{\to} N_1$ and
$M_2 \stackrel{t^2}{\to} M_2 \stackrel{t^{-2}}{\to} N_1$ since both paths compose to
$1$, as well as the difference between $N_1 \to R^\# \to M_1 \to N_1$ and
$N_1 \to M_2 \to N_1$ since both compose to $t^3$.

Next, we apply results from Section 3 to describe the minimal projective resolutions
of the simple $\Lambda$-modules.  Recall that these will become periodic of period
$2$ after the first two terms (since $m = \max \{2,d+1\} = 2$ here).  If $S(i)$ is a
simple left $\Lambda$-module, corresponding to a vertex $i$ of the quiver of
$\Lambda$, its minimal projective presentation is given by
$\bigoplus_{\alpha : j \to i} P(j) \stackrel{\pi(i)}{\to} P(i) \to S(i) \to 0$, where
the sum ranges over all arrows $\alpha$ ending at $i$, and the $\alpha$ component of
$\pi(i)$ is just the map $\alpha \colon P(j) \to P(i)$.  In general, we can find a map
$d(i)$ between $X, Y \in \add(\tilde{M})$ so that $\pi(i)$ is realized as
$\Hom_{R^\#}(\tilde{M},d(i)) \colon \Hom_{R^\#}(\tilde{M},X) \to \Hom_{R^\#}(\tilde{M},Y)$
and $\ker d(i) \in \MCM(R^\#)$ with
$\Hom_{R^\#}(\tilde{M},\ker d(i)) \cong \Omega^2_{\Lambda}(S(i))$.  Furthermore the
$\add(\tilde{M})$-resolution of $\ker d(i)$ will induce the remaining terms of the
minimal projective resolution of $S(i)$ over $\Lambda$.  For example, for $S(N_1)$,
the minimal projective presentation has the form
$P(M_1) \oplus P(M_2) \stackrel{\pi(N_1)}{\to} P(N_1) \to S(N_1) \to 0$ where
$\pi(N_1)$ is the map induced by
$\begin{pmatrix} 1 \\ t^{-2} \end{pmatrix} : M_1 \oplus M_2 \to N_1$, whose kernel is
isomorphic to $B$.  Using the $\Sigma_1$-approximations of $B$ and $A$, we now obtain
the minimal projective resolution
\[
\cdots \to P(M_1)^2 \oplus P(M_2) \to P(N_1)^2 \oplus P(M_2) \to
P(M_1) \oplus P(M_2) \to P(N_1) \to S(N_1) \to 0\,.
\]
Similarly for $S(M_2)$, we compute $\Omega^2_{\Lambda}(S(M_2)) \cong \Hom_{R^\#}(\tilde{M},\ker d(M_2))$, where 
\[
d(M_2) = \begin{pmatrix} t^3 \\ t^5 \\ t^2 \end{pmatrix} : M_1 \oplus N_1 \oplus M_2
\to M_2\,.
\]
One can compute $\ker d(M_2) \cong X$.  Then using the $\Sigma_1$-approximation sequence for $X$, we will get the minimal projective resolution
\[
\begin{split}
\cdots \to & P(M_1) \oplus P(N_1) \oplus P(M_2)^2 \to P(M_1) \oplus P(N_1) \oplus
P(M_2)^2 \to \\ & \hspace{2cm} P(M_1) \oplus P(N_1) \oplus P(M_2) \to P(M_2) \to S(M_2) \to 0.
\end{split}
\]
\end{example}

\begin{example}
Similar computations can be made for the $\EE_8$ curve singularity
$R^{\#} = k[\![x,y]\!]/(x^5+y^3) \cong k[\![t^3, t^5]\!]$.  We now set
$R = k[x]/(x^5)$, and see that in the notation of \cite{Yoshino:book} (except with $x$ and $y$ swapped)
\[
  \Omega_{R^\#}(R/(x)) \cong \cok((x^4,x) \hat{\otimes} (y,y^2)) = \cok
  \left( \begin{bmatrix} x^4 & -y \\ y^2 & x \end{bmatrix}, \begin{bmatrix} x & y \\
      -y^2 & x^4 \end{bmatrix} \right) \cong \cok(\psi_1, \phi_1) \cong N_1\,;
\]
 and
\[
\Omega_{R^\#}(R/(x^2)) \cong \cok((x^3,x^2) \hat{\otimes} (y,y^2)) = \cok
  \left( \begin{bmatrix} x^3 & -y \\ y^2 & x^2 \end{bmatrix}, \begin{bmatrix} x^2 & y \\
      -y^2 & x^3 \end{bmatrix} \right) \cong \cok(\psi_2, \phi_2) \cong N_2 ;
 \]
 from which it follows that $\Omega_{R^\#} (R/(x^3)) \cong \Omega_{R^\#}(N_2) \cong M_2$ and $\Omega_{R^\#} (R/(x^4)) \cong \Omega_{R^\#}(N_1) \cong M_1$.  Thus $\Sigma_1$
contains $R^\#$ along with the indecomposables
$M_1 \cong (t^3, t^{10}), N_1 \cong (t^3, t^5), M_2 \cong (t^6,t^{10})$ and
$N_2 \cong (t^5, t^6)$.  As before, the irreducible morphisms between these
indecomposable $R^{\#}$-modules can all be realized as multiplication by powers of
$t$, and we obtain the following quiver for $\Lambda$.
\[
  \xymatrixrowsep{3.0pc} \xymatrix{ M_2 \ar[d]<0.5ex>^{1} \ar@{=>}[rr]<0.5ex>^{1,
      t^{-1}} & & N_2 \ar@{=>}[ll]<0.5ex>^{t^4, t^5} \ar[d]<0.5ex>^1 \\ M_1
    \ar[u]<0.5ex>^{t^3} \ar[rr]<0.5ex>^1 & & N_1 \ar[u]<0.5ex>^{t^3}
    \ar[ll]<0.5ex>^{t^5} \ar[dl]^{1} \\ & R^\# \ar[ul]^{t^3}}
\] 
Thus, $\Lambda$ is isomorphic to a quotient of the completed path algebra of
the above quiver, with relations defined by the labels as in the previous example.

We also list the $\Sigma_1$-approximation sequences for the remaining indecomposable MCM $R^\#$-modules.  As in the previous example, for each $M \in \MCM(R^\#)$ the corresponding sequence is computed by first calculating $M/yM \in \MCM(R)$, which is easily done by looking at the matrix factorization associated to $M$.  We obtain the following sequences (and their syzygies):
\begin{align*}
& 0 \to B_1 \to M_1^2 \oplus N_2 \to A_1 \to 0 &\\
& 0 \to B_2 \to M_1 \oplus M_2^2 \to A_2 \to 0 & \\
& 0 \to D_i \to M_1 \oplus M_2 \oplus N_1 \oplus N_2 \to C_i \to 0 \ \ (i = 1,2) &\\
& 0 \to Y_1 \to M_1 \oplus M_2^2 \oplus N_1 \oplus N_2^2 \to X_1 \to 0 & \\
& 0 \to Y_2 \to M_1^2 \oplus M_2 \oplus N_2^2 \to X_2 \to 0 &
\end{align*}

Finally, we mention that the short exact sequences realizing each indecomposable
$M \in \MCM(R^\#)$ as a direct summand of an extension of modules in $\Sigma_1$ are
not as apparent here as they were in our previous example.  Here, only the modules
$A_1, B_1, C_2$ and $D_2$ arise in the middle terms of almost split sequences ending
with objects in $\Sigma_1$.  We sketch the construction of this sequence in one other
example.  Consider the module
$$A_2 \cong \cok \begin{bmatrix} x & -y & 0 \\ 0 & x^2 & -y \\ y & 0 &
  x^2 \end{bmatrix}.$$ As in the proof of Proposition~\ref{prop:SigmastarSigma}, we
obtain the desired short exact sequence by applying $\Omega_{R^\#}$ to the sequence
$$0 \to yA_2/y^2A_2 \to A_2/y^2A_2 \to A_2/yA_2 \to 0.$$
An easy computation shows that both $yA_2/y^2A_2$ and $A_2/yA_2$ are isomorphic to $R/(x) \oplus (R/(x^2))^2$ as $R$-modules.  Hence, using Proposition~\ref{prop:newHP} we obtain the short exact sequence 
$$0 \to N_1 \oplus N_2^2 \to A_2 \oplus B_2 \oplus F \to N_1 \oplus N_2^2 \to 0$$
for some free module $F$.  Since $A_2, N_1$ and $N_2$ have rank $1$, while $B_2$ has rank $2$, we have $F \cong (R^\#)^3$.
\end{example}

\section*{Acknowledgments}
This work was begun during the authors' participation in the research program
``IRTATCA: Interactions between Representation Theory, Algebraic Topology and
Commutative Algebra'', at the Centre de Recerca Matem\`atica in Barcelona in
2015. The authors thank the CRM for a pleasant and productive experience.  GJL was
supported by NSF grant DMS-1502107, and ASD was supported by NSF conference grant DMS-58502086.

The authors are grateful to Tokuji Araya for pointing out a gap in an earlier version
of this paper.

\bibliographystyle{amsalpha}
%\bibliography{Refs-2017}

\def\cprime{$'$} \def\polhk#1{\setbox0=\hbox{#1}{\ooalign{\hidewidth
  \lower1.5ex\hbox{`}\hidewidth\crcr\unhbox0}}}
  \def\polhk#1{\setbox0=\hbox{#1}{\ooalign{\hidewidth
  \lower1.5ex\hbox{`}\hidewidth\crcr\unhbox0}}}
\providecommand{\bysame}{\leavevmode\hbox to3em{\hrulefill}\thinspace}
\providecommand{\MR}{\relax\ifhmode\unskip\space\fi MR }
\providecommand{\MRhref}[2]{%
  \href{http://www.ams.org/mathscinet-getitem?mr=#1}{#2}
}
\providecommand{\href}[2]{#2}

\end{document}